\documentclass[11pt, reqno]{amsart}

 \usepackage[foot]{amsaddr}
\usepackage[latin1]{inputenc}
\usepackage{amsfonts}
\usepackage{amsmath}
\usepackage{amssymb}
\usepackage{amsthm}
\usepackage{fontenc}
\usepackage{graphicx}
\usepackage[margin=1.1in]{geometry}
\usepackage[all]{xy}
\usepackage{hyperref}

\newtheorem{theorem}{Theorem}[section]
\newtheorem{lemma}[theorem]{Lemma}
\newtheorem{proposition}[theorem]{Proposition}

\newtheorem{corollary}[theorem]{Corollary}
\newtheorem{definition}{Definition}
\newtheorem{remark}{Remark}

\newtheorem{conjecture}{Conjecture}
\newtheorem{question}{Question}

\numberwithin{equation}{section}

\DeclareMathOperator{\Conv}{Conv}

\DeclareMathOperator{\dist}{dist} 

\begin{document}

\title{Regular Polygonal Partitions of a Tverberg Type}

\author{Leah Leiner}
\author{Steven Simon} 
\address{Bard College, Department of Mathematics}
\email{lleiner@bard.edu}
\email{ssimon@bard.edu}

\begin{abstract}  A seminal theorem of Tverberg states that any set of $T(r,d)=(r-1)(d+1)+1$ points in $\mathbb{R}^d$ can be partitioned into $r$ subsets whose convex hulls have non-empty $r$-fold intersection.  Almost any collection of fewer points in $\mathbb{R}^d$ cannot be so divided, and in these cases we ask if the set can nonetheless be $P(r,d)$--partitioned, i.e., split into $r$ subsets so that there exist $r$ points, one from each resulting convex hull, which form the vertex set of a prescribed convex $d$--polytope $P(r,d)$. Our main theorem shows that this is the case for any generic $T(r,2)-2$ points in the plane and any $r\geq 3$ when $P(r,2)=P_r$ is a regular $r$--gon, and moreover that $T(r,2)-2$ is tight. For higher dimensional polytopes and $r=r_1\cdots r_k$, $r_i \geq 3$, this generalizes to $T(r,2k)-2k$ generic points in $\mathbb{R}^{2k}$ and orthogonal products $P(r,2k)=P_{r_1}\times \cdots \times P_{r_k}$ of regular polygons, and likewise to $T(2r,2k+1)-(2k+1)$ points in $\mathbb{R}^{2k+1}$ and the product polytopes $P(2r,2k+1)=P_{r_1}\times \cdots \times P_{r_k} \times P_2$. As with Tverberg's original theorem, our results admit topological generalizations when $r$ is a prime power, and, using the ``constraint method" of Blagojevi\'c, Frick, and Ziegler, allow for dimensionally restricted versions of a van Kampen--Flores type and colored analogues in the fashion of Sober\'on.\end{abstract} 
\maketitle

\section{Introduction and Statement of Main Results} 

Tverberg's landmark 1966 theorem [28] states that any set of $T(r,d):=(r-1)(d+1) +1$ points in $\mathbb{R}^d$ can be divided into $r$ pairwise disjoint subsets whose convex hulls have non-empty $r$-fold intersection (called a Tverberg $r$-partition). The $r = 2$ case recovers Radon's Theorem, and the result also has deep connections with the classical theorems of Helly and Carath\'eodory. We refer the reader to the recent surveys [5, 9, 10] for a sampling of the many interesting applications and extensions of Tverberg's theorem in discrete geometry, combinatorics, topology, and beyond. 

	For codimension reasons, almost any collection of $N$ points fails to admit a Tverberg $r$-partition when $N$ is less than the Tverberg number $T(r,d)$. In these cases one may consider weaker  conditions on the convex hulls arising from partitions by $r$ subsets. The most studied problem in this direction (see, e.g., [1, 20]) was initiated by Reay [21], and  for each $2\leq j<r$ asks whether there exists some $N<T(r,d)$ such that any generic set of $N$ points in $\mathbb{R}^d$ can be partitioned into $r$ subsets such that while any $j$ of the resulting convex hulls have common intersection, all $r$ of them do not. In fact, Reay's ``relaxed" conjecture  claims that no such $N$ exists, even when $j=2$. Instead of partial intersection conditions on convex hulls, we introduce the following polytopal variant with a slight Ramsey flavor: 

\begin{question} \label{quest1} Let $n\leq d$ and let $P(r,n)$ be a $n$--dimensional convex polytope with $r$ vertices. What is the minimum $N:=N_{(P(r,n);d)}$ such that almost any $N$ points in $\mathbb{R}^d$ be partitioned into $r$ sets $A_1,\ldots, A_r$ such that there exist $r$ points $x_1\in \Conv(A_1),\ldots, x_r\in \Conv(A_r)$ which are the vertices of a convex polytope similar to $P(r,n)$?  \end{question} 
 
 Any partition as in Question \ref{quest1} will be called a \textit{$P(r,n)$--partition}, or simply a \textit{polytopal partition} if the context is clear. We denote $N_{(P(r,n);d)}$ by $N_{P(r,d)}$ in the special case that $n=d$. Our central result determines $N_{P(r,2)}$ for all regular polygons $P_r=P(r,2)$ in terms of the Tverberg number:
 
 \begin{theorem}\label{thm1.1} $N_{P_r}=T(r,2)-2=3r-4$ for all regular $r$-gons $P_r$.\end{theorem} 	 
 
 The figure below gives the two possible cases of Theorem \ref{thm1.1} when $r=3$, in which either one or two of the vertices of the equilateral triangle come from a given set of 5 (blue) points: 
\begin{center} \includegraphics[scale=.55]{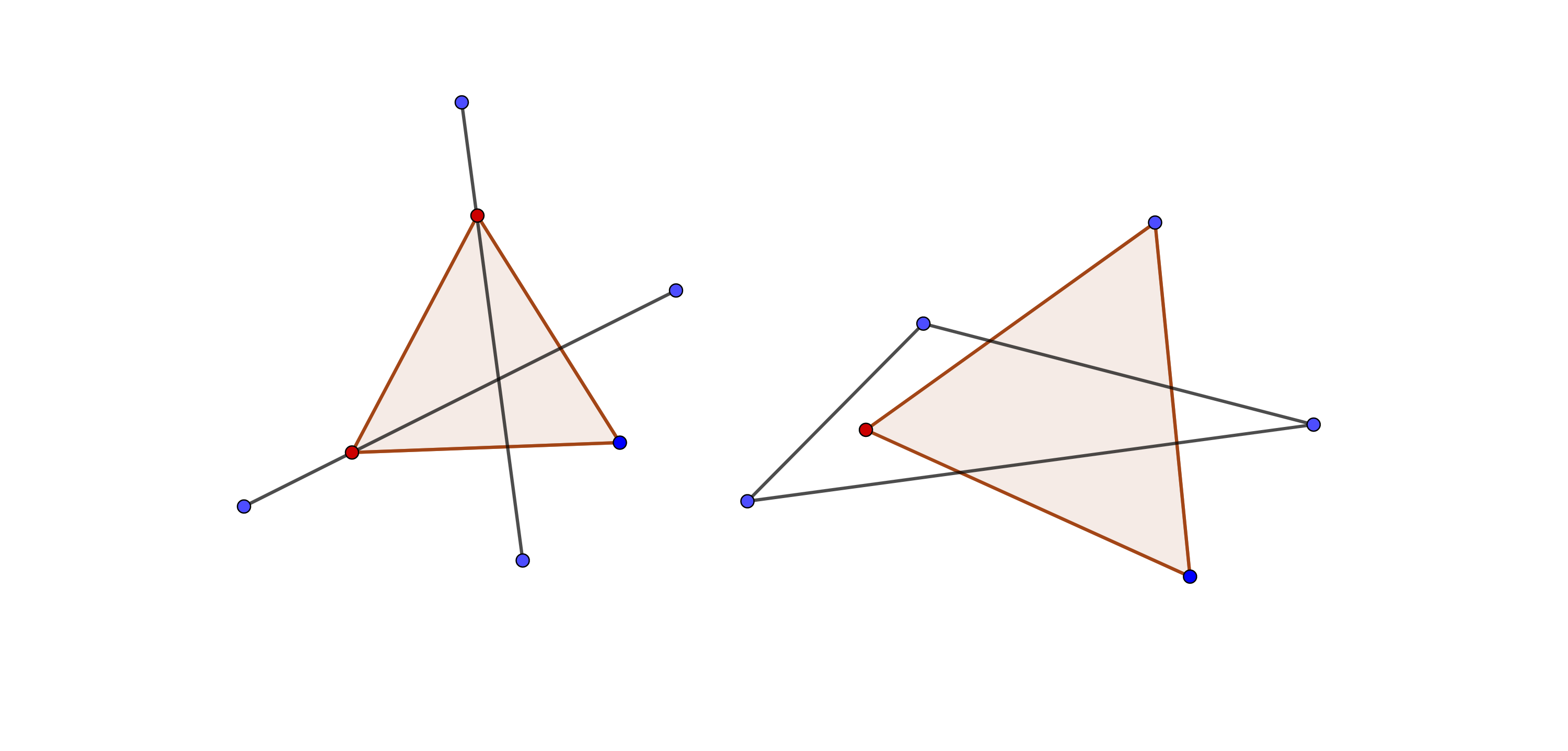}\end{center} 
\vspace*{-.275in} 

For higher dimensional polytopes, we give upper bounds for $N_{P(r,d)}$ in the case of ``multiprisms", i.e., the Cartesian products \begin{equation}\label{eq:1.1} P(r,2k):=P_{r_1}\times \cdots \times P_{r_k} \end{equation} of $k$ pairwise orthogonal (two--dimensional) regular $r_i$-gons in $\mathbb{R}^{2k}$, $r_i\geq 3$, as well as for the pairwise orthogonal products 
								 \begin{equation}\label{eq:1.2} P(2r,2k+1):=P_{r_1}\times \cdots \times P_{r_k} \times P_2 \end{equation} in $\mathbb{R}^{2k+1}$. Here $P_2$ denote a line segment, so that for  example $P(2r,3)=P_r\times P_2$ is a right regular prism in $\mathbb{R}^3$.
								 
\begin{theorem}\label{thm1.2} Let $k\geq 1$ and let $r=r_1\cdots r_k$, $r_i\geq 3$. Let $P(r,2k)=P_{r_1}\times \cdots \times P_{r_k}$ and $P(2r,2k+1)=P_{r_1}\times \cdots \times P_{r_k} \times P_2$. Then \begin{equation}\label{eq:1.3}  N_{P(r,2k)}\leq T(r,2k)-2k\,\,\, \text{and}\,\,\, N_{P(2r,2k+1)}\leq T(2r,2k+1)-2k-1.\end{equation}
\noindent Moreover,\\
(a) Let $S$ be a set of $T(r,2k)-2k$ generic points in $\mathbb{R}^{2k}$. For any decomposition of $\mathbb{R}^{2k}$ by $k$ pairwise orthogonal coordinate planes $U_1,\ldots, U_k$,  there exists a $P(r,2k)$--partition of $S$ such that $P_{r_i}$ is parallel to $U_i$ for all $1\leq i\leq k$. Likewise,\\
(b) Let $S$ be a set of $T(2r,2k+1)-2k-1$  generic points in $\mathbb{R}^{2k+1}$. For any $k$ pairwise orthogonal coordinate planes $U_1,\ldots, U_k$ in $\mathbb{R}^{2k+1}$, there exists a $P(2r,2k+1)$--partition of $S$ such that $P_{r_i}$ is parallel to $U_i$ for all $1\leq i\leq k$.
\end{theorem}

 	For example, Theorem \ref{thm1.2}(b) guarantees that for any generic set of $8r-6$ points in $\mathbb{R}^3$, there exist $2r$ points, one form each convex hull determined by a  partition by $2r$ subsets, which are the vertices of a right regular prism whose $r$-gon base can be prescribed parallel to any of the three coordinate planes.\\
	
	Although parts (a) and (b) of Theorem \ref{thm1.2} are tight under the given coordinate conditions, this should not be the case for the upper bound (\ref{eq:1.3}) itself. To see this, consider the $2(d-2)$-dimensional  Grassmanian manifold $G_2(\mathbb{R}^d)$ of all linear 2-flats in $\mathbb{R}^d$. Allowing non--coordinate multiprisms, there are $\dim \Pi_{i=0}^{k-1} G_2(\mathbb{R}^{2k-2i})=2k(k-1)$ remaining degrees of freedom for the existence of any $P(r,2k)$--partition, and subtracting this from $N=T(r,2k)-2k$ above for coordinate multiprisms yields an expected value of $N_{P(r,2k)}=T(r,2k)-2k^2$. Similar remarks lead one to expect $N_{P(2r,2k+1)} =T(2r,2k+1)-2k(k+1)-1$.
\\	
	
	As with Tverberg's theorem, topological extensions of Question \ref{quest1} arise by viewing $N$ points in $\mathbb{R}^d$ as the image of the vertices of the $(N-1)$-simplex  $\Delta_{N-1}$ under a map $f:\Delta_{N-1}\rightarrow \mathbb{R}^d$. Considering affine linear maps, a Tverberg $r$-partition of a set of $N$ points is equivalent to the existence of $r$ pairwise disjoint faces $\sigma_1,\ldots, \sigma_r \subseteq \Delta_{N-1}$ such that $\cap_{i=1}^r f(\sigma_i)\neq \emptyset$. Likewise, a $P(r,n)$-partition is equivalent to finding points $x_1\in \sigma_1\,\ldots, x_r\in \sigma_r$ from pairwise disjoint $\sigma_i$ so that the $f(x_i)$ are the vertices some $P(r,n)$. It is the content of the Topological Tverberg theorem of [19] (see also [23, 29]) that affine maps may be replaced by arbitrary continuous ones for all prime powers $r$ when $N=T(r,d)$. On the other hand, counterexamples to such topological extensions of Tverberg's theorem for all other $r$ were recently produced [6, 12] based on  fundamental work of [15]. We refer the reader to the reviews of [5, 9, 10] amongst others for the history of this very famous problem initially raised in [4]. By analogy with the Topological Tverberg theorem, we have topological $P(r,2k)$--partitions for all odd prime powers $r=p^k$, as well as for $k$--orthotopes $P(2^k,k):=P_2^{\times k}$ in $\mathbb{R}^k$. 
			
\begin{theorem}\label{thm1.3} $\newline$
(a) Let $N=T(p^k,2k)-2k$, $p$ an odd prime, and let $f: \Delta_{N-1}\rightarrow \mathbb{R}^{2k}$ be a generic continuous map. For any decomposition of $\mathbb{R}^{2k}$ by $k$ pairwise orthogonal coordinate planes $U_1,\ldots, U_k$, there exist points $x_1\in \sigma_1,\ldots, x_r \in\sigma_{p^k}$ from pairwise disjoint faces so that  $f(x_1),\ldots, f(x_{p^k})$ form the vertex set of a  multiprism $P(r,2k)=P_p^{\times k}$ whose regular $p$-gons are parallel to the $U_i$. This also holds for $r=4$ when $k=1$.\\
(b)  Let $N=T(2^k,k)-k$ and let $f: \Delta_{N-1}\rightarrow \mathbb{R}^k$ be a generic continuous map.  Then there exist points $x_1\in \sigma_1,\ldots, x_{2^k}\in \sigma_{2^k}$ from pairwise disjoint faces so that $f(x_1),\ldots, f(x_{2^k})$ are the vertices of a $k$-orthotope $P_2^{\times k}$ whose edges are parallel to the coordinate axes. 
\end{theorem} 

See Definition \ref{def1} below for our precise notion of topological (and affine) generiticity. As discussed in Remark \ref{rem1} following this definition, our generiticity condition holds for almost every affine map and remains typical in the continuous setting. In particular, it includes all maps $f:\Delta_{N-1}\rightarrow \mathbb{R}^d$ with $N<T(r,d)$ which do not admit a Tverberg $r$-partition.\\

	The remainder of this paper proceeds as follows. Theorems 1.1--1.3 have equivalent formulations in terms of Fourier analysis on finite abelian groups, so that the existence of prescribed $P(r,d)$--partitions are equivalent to the annihilation of certain Fourier coefficients.  This perspective was first introduced to Tverberg--type problems by the second author in [25] and is discussed in Section 2. In Section 3, we show that nearly arbitrarily prescribed coefficients can be forced to vanish in the affine setting (Theorem \ref{thm3.1}) provided generically tight dimensional considerations are met. This is ultimately derived as a consequence of Sarkaria's ``Linear Borsuk--Ulam Theorem" [23], which is itself a corollary of B\'ar\'any's colored Carath\'eodory Theorem [2]. As we show in Section 4, Theorem \ref{thm1.2} quickly follows, as do two tight extensions of Theorem \ref{thm1.1} to higher dimensions (Theorem \ref{thm4.1}): one for parametrized $P_r$--partitions in all dimensions and the other for $P_r$--partitions in even dimensions for which $P_r$ is constrained to lie in a complex 1--flat. In particular, one has  the upper bound $N_{(P_r;d)}\leq T(r,d)-d$ for even $d$. 
	
	While one must be careful to guarantee the vanishing of desired Fourier coefficients when $f$ is continuous (see, e.g., the polynomial criteria [25, Theorem 3.2]), coefficients can be annihilated just as freely as in the affine setting provided one considers elementary abelian groups of prime power order, so that Theorem \ref{thm1.3} follows from standard equivariant cohomological techniques discussed in Section 5. In Section 6, we give two examples of how the ``constraint" method of [7] (and implicitly [13]) can be applied to our framework. These produce regular polygonal partitions (i) of a van Kampen--Flores type for odd primes (Theorem \ref{thm6.2}), including when $N=T(r,d)$ (Corollary \ref{cor6.3}),  as well as (ii) a colored variant in the mode of Sober\'on [23] (Theorem \ref{thm6.5}). For (i), we observe a connection between our results and those surrounding the famous square peg problem initiated by Toeplitz [27], while the $r=3$ case of (ii) matches the dimension of the B\'ar\'any--Larman Conjecture [3]. We conclude in Section 7 with a return to Question \ref{quest1} when $P(r,r-1)=\Delta_{r-1}$ is a regular $(r-1)$-simplex. While standard methods yield $N_{(\Delta_2;d)}=5$ for all $d\geq2$ (Proposition \ref{prop7.1}), an observation we owe to Florian Frick, we show that this approach fails for all $r\geq 4$, including when $r$ is a prime power.

\section{A Fourier Analytic Approach} 

Following [25], Fourier analysis on finite groups can be applied to any map $f=(f_1,\ldots, f_d) :\Delta_N\rightarrow \mathbb{C}^d$ (including when any of the $f_i$ are real--valued) by indexing each collection of $r$ points $x_1\in \sigma_1,\ldots, x_r\in \sigma_r$ from pairwise disjoint faces by a fixed group $G$ of order $r$. For each coordinate map $f_i$ and any such $\{x_g\}_{g\in G}\subset \Delta_N$, evaluation of $f_i$ defines the function

 \begin{equation} \label{eq:2.1} F_i: G\rightarrow \mathbb{C}, \,\,\, g\mapsto f_i(x_g),\end{equation} 
  which has a Fourier decomposition arising from the complex representation theory of $G$. When $G=\oplus_{j=1}^k \mathbb{Z}_{r_j}$ is abelian, this takes a simple form owing to the fact that the irreducible representations are all one-dimensional and indexed by the group itself. Explicitly, each $\chi_h: G\rightarrow \mathbb{C}^\times$ is given by $\chi_h(g)=\Pi_{j=1}^k \zeta_{r_j}^{h_jg_j}$, where $h=(h_1,\ldots, h_k), g=(g_1,\ldots, g_k) \in G$ and $\zeta_r=e^{2\pi i/r}$ is the standard $r$-th root of unity. The characters $\chi_h$ form an orthonormal basis for the space of all functions $H: G\rightarrow \mathbb{C}$ under the standard inner product $\langle H_1, H_2\rangle = \frac{1}{|G|}\sum_{g\in G} H_1(g)\overline{H_2}(g)$ for each $H_1,H_2: G\rightarrow \mathbb{C}$ (see, e.g., [24]), so that each $F_i$ above can be uniquely expressed as 
  
				\begin{equation} \label{eq:2.2} F_i =\sum_{h\in G} c_{i,h} \chi_h, \end{equation} where
				 \begin{equation}\label{eq:2.3} c_{i,h}=\langle F_i, \chi_h\rangle=\frac{1}{|G|}\sum_{g\in G} f_i(x_g)\chi_h^{-1}(g)\end{equation} is the Fourier coefficient corresponding to $\chi_h$.\\
				
	As $\chi_0=1$, it follows immediately from (\ref{eq:2.2}) that $F_i$ is constant iff $c_{i,h}=0$ for all $h\neq 0$. In particular, a Tverberg $r$-partition is equivalent to the vanishing of all Fourier coefficients not arising from the trivial representation.\\
	
	For the multiprism partitions of Theorem \ref{thm1.2} and \ref{thm1.3}(a), we will see in Section 4 that it suffices to consider the special case where the $U_1,\ldots, U_k$ are the coordinate complex planes of $\mathbb{C}^k$, the Fourier characterization of which is given below (the situation for Theorem 1.3(b) is analogous, see Section 5). As usual, we let $\mathbb{Z}_r^\times$ denote the units of the ring $\mathbb{Z}_r$ and $\mathbf{e}_i$  the $i$-th standard basis vector of $G=\mathbb{Z}_{r_1}\oplus \cdots \oplus \mathbb{Z}_{r_k}$.

 \begin{lemma}\label{lem2.1} Let $r=r_1,\ldots, r_k$, $r_i\geq 3$, let $G=\oplus_{i=1}^k \mathbb{Z}_{r_i}$, and let $G'=G\oplus \mathbb{Z}_2$.\\
(a) Let $f: \Delta_N\rightarrow \mathbb{C}^k$, and let $\{x_g\}_{g\in G}\subset \Delta_N$ be a collection of points from pairwise disjoint faces. Then $\{f(x_g)\}_{g\in G}$ is the vertex set of a multiprism $P(r,2k)=P_{r_1}\times \cdots \times P_{r_k}$ with each $P_{r_i}$ parallel to the $i$-th coordinate plane of $\mathbb{C}^k$ iff for each $1\leq i \leq k$, there exists some $g_{i_0}\in \mathbb{Z}^\times_{r_i}$ such that  (i) $c_{i,g_{i_0}\mathbf{e}_i}\neq 0$ and (ii) $c_{i,g}=0$ for all $g\in G-\{0, g_{i_0}\mathbf{e}_i\}$ in the Fourier expansion (\ref{eq:2.2}). Likewise,\\
(b)  Let $f: \Delta_N\rightarrow \mathbb{C}^k\oplus\mathbb{R}$, and let $\{x_{g'}\}_{g'\in G'}\subset\Delta_N$ be a collection of points from pairwise disjoint faces. Then $\{f(x_{g'})\}_{g'\in G'}$ is the vertex set of a multiprism $P(2r, 2k+1)=P_{r_1}\times \cdots \times P_{r_k}\times P_2$ with each $P_{r_i}$ parallel to the $i$-th coordinate plane of $\mathbb{C}^k$ iff for each $1\leq i \leq k+1$ there exists some $g'_{i_0}\in \mathbb{Z}^\times_{r_i}$ such that  (i) $c_{i,g'_{i_0}\mathbf{e}_i}\neq 0$ and (ii) $c_{i,g'}=0$ for all $g'\in G'-\{0, g_{j'_0}\mathbf{e}_i\}$ in the Fourier expansion (\ref{eq:2.2}).
\end{lemma} 

\begin{proof} For part (a), consider some collection $\{x_g\}_{g\in G}$ from pairwise disjoint faces. The $f(x_g)=(f_1(x_g),\ldots, f_k(x_g))$ are the vertices of a $P(r,2k)$ with each $P_{r_i}$ parallel to the $i$-th coordinate plane iff  each $\{f_i(x_g)\}_{g\in G}$ is a regular $r_i$-gon in $\mathbb{C}$. For each $1\leq i \leq k$, the Fourier decomposition (\ref{eq:2.2}) of each $f_i$ gives $f_i(x_g)=c_0 + \sum_{h\neq 0} c_{i,h} \chi_h(g)$. As $\chi_{g_{i_0}\mathbf{e}_i}(g)=\zeta_{r_i}^{g_{i_0}g_i}$, one has $f_i(x_g)=c_{i,0}+ c_{i, g_0}\zeta_{r_i}^{g_{i_0}g_i}$ if $c_{i,g}=0$ for all $g\in G-\{0,g_{i_0}\mathbf{e}_i\}$.  If $g_{i_0}\in \mathbb{Z}_{r_i}^\times$, then $\zeta_{r_i}^{g_{i_0}}$ is a primitive $r_i$-th root of unity, and hence $\{f_i(x_g)\}_{g\in G}$ is the vertex set of a regular $r_i$-gon, provided in addition that $c_{g_{i_0}\mathbf{e}_i}\neq 0$. Conversely,  it follows from the orthogonality of characters that  if $\{f_i(x_g)\}_{g\in G}$ is the vertex set of a regular $r_i$-gon, then $f_i(x_g)=c_{i,0} +c_{i, g_{i_0}\mathbf{e}_i}\zeta_{r_i}^{g_{i_0}g_i}$ for some $g_{i_0}\in \mathbb{Z}_{r_i}^\times$. Thus all Fourier coefficients other than $c_{i,0}$ and $c_{i,g_{i_0}\mathbf{e}_i}$ vanish, and $c_{g_{i_0}\mathbf{e}_i}\neq 0$.

	The proof for the multiprisms $P(2r,2k+1)=P(r,2k)\times P_2$ of part (b) is identical, except now $\{f_{k+1}(x_{g'})\}_{g'\in G}\subset \mathbb{R}$ are the endpoints of a segment.
\end{proof}

\section{Annihilating Coefficients in the Affine Setting} 

Theorem \ref{thm1.2} follows once it is ensured that the coefficient conditions prescribed in Lemma \ref{lem2.1} are met. This will follow from Theorem \ref{thm3.1} below, which in essence guarantees that outside of trivial restrictions any collection of prescribed Fourier coefficients can be simultaneously annihilated once the dimension of the simplex is sufficiently large. 

	Before providing the formal statement of Theorem \ref{thm3.1}, we provide some motivation and intuition. Namely, suppose that $F_i: G\rightarrow \mathbb{C}$ is as in (\ref{eq:2.1}) above, and that, for a given a subset $S_i$ of $G$, we wish to annihilate all  coefficients $c_{i,h}$ in the Fourier expansion (\ref{eq:2.2}) with $h$ from $S_i$.  This will be shown to possible so long as each $h$ is non-zero and $N$ is large enough. If $f_i$  is real--valued, one should additionally consider whether or not some of the $h$ have order two. This consideration follows from the easily verified fact that $\overline{c_{i,h}}=c_{i,-h}$ for any $h\in G$ when $f_i$ is real--valued, so that $c_{i,h}$ vanishes iff $c_{i,-h}$ does. In order to avoid redundancy in annihilation, and in particular to ensure the optimality of $N$ below, we therefore stipulate below that $-h$ be excluded from $S_i$ whenever $h$ is in $S_i$ and $|h|>2$. Of course, no redundancy arises when $|h|=2$. Additionally, observe that each Fourier coefficient $c_{i,h}$ is real--valued when $f_i$ is real--valued and $h$ has order two.
		  
\begin{theorem}\label{thm3.1} Let $f:\Delta_N\rightarrow \mathbb{C}^d\oplus \mathbb{R}^{d'}$ be an affine map, where $d+d'\geq 1$ and $d,d'\geq 0$. Let $G=\mathbb{Z}_{r_1}\oplus \cdots \oplus \mathbb{Z}_{r_k}$, $r_i\geq 2$, and let $r=r_1\cdots r_k$. For each $1\leq i \leq d+d'$, suppose that $S_i\subseteq G-\{0\}$ and let $m_i=|S_i|$. For each $d+1\leq i\leq d+d'$, let $T_i=\{h\in S_i\mid |h|=2\}$ and let $|T_i|=m'_i$. Moreover, suppose that each $S_i$ has the property that $-h\notin S_i$ if $h\in S_i-T_i$. Finally, let $m=\sum_{i=1}^{d+d'}m_i$ and let $m'=\sum_{i=d+1}^{d+d'} m_i'$. If $N=2m-m'+ r-1$, then there exists some $\{x_g\}_{g\in G}\subset\Delta_N$ with pairwise disjoint support such that $c_{i,h}=0$ in the Fourier expansions (\ref{eq:2.2})  for all $h\in S_i$ and all $1\leq i \leq d+d'$. Moreover, the existence of some $\{x_g\}_{g\in G}$ with pairwise disjoint support such that $c_{i,h}=0$ in (\ref{eq:2.2}) for all $h\in S_i$ and all $1\leq i \leq d+d'$ fails for almost any affine map $f$ if $N<2m-m'+r-1$.  
\end{theorem} 

.

	Our proof of Theorem \ref{thm3.1} relies on the join configurations commonly used in Tverberg--type problems. Recall that the $r$-fold join $\Delta_N^{\ast r}$ of $\Delta_N$ consists of all the formal convex sums $\lambda_1x_1+\cdots + \lambda_rx_r \in \sigma_1\ast \cdots \ast \sigma_r$ coming from the $r$--fold join of any $r$ faces of $\Delta_N$, including the possibility of empty faces. The deleted $r$-fold join \begin{equation}\label{eq:3.1} (\Delta_N)^{\ast r}_\Delta = \{\lambda_1x_1+\cdots + \lambda_rx_r\in \sigma_1\ast \cdots \ast \sigma_r \mid x_i\in \sigma_i\,\,\,\, \text{and}\,\,\, \sigma_i\cap \sigma_j=\emptyset\,\, \forall\,\, i\neq j\} \end{equation} is the subcomplex consisting of all formal convex sums from the joins of pairwise disjoint faces. Following Sarkaria [23], we parametrize each $r$-tuple of disjoint faces by $G$ and denote the resulting complex by $(\Delta_N)^{\ast G}_\Delta$. The group acts freely on $(\Delta_N)^{\ast G}_\Delta$ by right translations, so that $g'\cdot \sum_g\lambda_gx_g=\sum_g\lambda_{g-g'}x_{g-g'}$. On the other hand, one can parametrize $r$ disjoint copies of the vertex set $\{v_1,\ldots, v_{N+1}\}$ of $\Delta_N$ by $G$ as well, so that $v_j^g$ will denote the $g$-th copy of $v_j$. One then has an equivalent parametrization of all points from the join of pairwise disjoint faces by the $(N+1)$-fold join \begin{equation}\label{eq:3.2} G^{\ast (N+1)}=\left \{\sum_{j =1}^{N+1} t_j v_j ^{g_j} \mid g_j \in G\,\,\, \text{and}\,\,\, \sum_{j=1}^{N+1} t_j=1,\, t_j\geq0\right\} \end{equation} of the group itself. As with $(\Delta_N)^{\ast G}_\Delta$, $G$ also acts freely on $G^{\ast (N+1)}$, now by affine extension of the $G$--action on each copy of $G$ in $G^{\ast (N+1)}$ by addition: $g'\cdot \sum_{j=1}^{N+1} t_j v_j^{g_j} =\sum_{j=1}^{N+1} t_j v_j^{g'+g_j}$. Finally, there is an obvious isomorphism $\iota: (\Delta_N)^{\ast G}_\Delta \cong G^{\ast (N+1)}$ between these two simplicial complexes obtained by grouping. Explicitly, for each $v\in G^{\ast (N+1)}$ and each $g\in G$, consider the sets $J_g=\{j\mid g_j=g \,\, \text{and}\,\, t_j>0\}$ for each $g\in G$. One then defines $\iota(v)=\sum_{g\in G} \lambda_g x_g$, where \begin{equation} \label{eq:3.3} \lambda_g=\sum_{j\in J_g} t_j\,\, \text{and}\,\, x_g=\sum_{j\in J_g}\frac{t_j}{\lambda_g} v_j \end{equation} if $J_g\neq \emptyset$, and (2) $\lambda_g=0$ otherwise. Moreover, it is easily seen that $\iota$ respects the two $G$ actions.\\ 

	Theorem \ref{thm3.1} will follow from an application of Sarkaria's ``Linear Borsuk--Ulam" theorem [23, Theorem 2.4]:\\

\begin{theorem}\label{thm3.2} Let $W$ be a real $N$-dimensional linear representation of $G$ which does not contain the trivial subrepresentation. If $A: G^{\ast (N+1)} \rightarrow W$ is an affine linear $G$-equivariant map, then there exists some $v\in G^{\ast (N+1)}$ such that $A(v)=0$.\end{theorem}

\begin{proof}[Proof of Theorem \ref{thm3.1}] Let $f:\Delta_N\rightarrow \mathbb{C}^d\oplus \mathbb{R}^{d'}$ with $N=2m-m'+r-1$. To apply Theorem \ref{thm3.2}, we construct a linear $G$-equivariant map $A$ whose zeros correspond to the vanishing of the prescribed Fourier coefficients. To that end, let $S_i$, $T_i$, $m$, and $m'$ be as in the statement of Theorem \ref{thm3.1} and consider the direct sum \begin{equation}\label{eq:3.4} \sigma=\underset{\stackrel{1\leq i \leq d+d'}{h\in S_i}}{\oplus} \chi_{-h}. \end{equation} Each character $\chi_h$ is real--valued iff $h\in T_i$, so $\sigma: G\rightarrow \mathbb{C}^{m-m'}\oplus \mathbb{R}^{m'}$.

To rule out empty faces, we must also consider $\mathbb{R}^\perp[G]=\{(\lambda_g)_{g\in G} \mid \sum_{g\in G}\lambda_g=0, \lambda_g\in \mathbb{R}\}$, the orthogonal complement of the trivial subrepresentation of the regular representation $\mathbb{R}[G]$. The action here is again by right translation, so that each $g'\in G$ sends each $(\lambda_g)_{g \in G}$ to $(\lambda_{g-g'})_{g\in G}$. It follows that \begin{equation}\label{eq:3.5} W:=\mathbb{C}^{m-m'}\oplus\mathbb{R}^{m'}\oplus \mathbb{R}^\perp[G] \end{equation} is a real $N$-dimensional representation which does not contain the trivial subrepresentation.\\

	We define $A=\mathcal{A}\circ \iota: G^{\ast (N+1)}\rightarrow W$, where  \begin{equation}\label{eq:3.6} \mathcal{A}=\oplus_{i, h\in S_i} \mathcal{F}_{i,h} \oplus \mathcal{R}: (\Delta_N)^{\ast  G}_\Delta \rightarrow \mathbb{C}^{m-m'} \oplus \mathbb{R}^{m'} \oplus \mathbb{R}^\perp [G] \end{equation} is given by

			\begin{equation}\label{eq:3.7} \mathcal{F}_{i,h}(\sum_{g\in G} \lambda_gx_g) = \sum_{g\in G} \lambda_g f_i(x_g)\chi_h^{-1}(g)  \end{equation} for all $h\in S_i$ and all $1\leq i \leq d+d'$ and
					\begin{equation}\label{eq:3.8} \mathcal{R}(\sum_{g\in G} \lambda_gx_g) = (\lambda_g-\frac{1}{|G|})_{g\in G}. \end{equation} Observe that $\mathcal{F}_{i,h}$ is real--valued when $h\in T_i$ and $i>d$.\\ 
					
	As $\mathcal{R}(\lambda\, x)=0$ iff $\lambda_g=\frac{1}{|G|}$ for all $g\in G$, we see that the zeros of $A$ correspond to those $\{x_g\}_{g\in G}$ from non-empty pairwise disjoint faces of $\Delta_N$ for which the Fourier coefficients $c_{i,h}$ in (\ref{eq:2.2}) vanish for all $h\in S_i$ and all $1\leq i \leq d+d'$. As $\chi^{-1}_h=\chi_{-h}$, $\mathcal{A}$ is equivariant with respect to the actions  on $(\Delta_N)^{\ast G}$ and $W$, so that $A$ is equivariant with respect with respect to the actions on $G^{\ast (N+1)}$ and  $W$ as well.\\

	It remains to check that $A$ is affine. For any $i$ and any $h\in S_i$, we have $(\mathcal{F}_{i,h}\circ \iota)(v_j^g)=f_i(v_j)\chi_h^{-1}(g)$ for each $g\in G$. Let $v=\sum_{j=1}^{N+1}t_j v_j^{g_j}$ and let $\iota (v)=\sum_g \lambda_g x_g$ as above. For those $g\in G$ with $J_g\neq \emptyset$, we have $f_i(x_g)=\frac{1}{\lambda_g}\sum_{j\in J_g} t_jf_i(v_j)$ because $f$ is affine. It follows that $(\mathcal{F}_{i,h}\circ \iota)(v)=\sum_g \lambda_g f_i(x_g)\chi_h^{-1}(g)=\sum_g \sum_{j\in J_g} t_jf_i(v_j)\chi_h^{-1}(g)=\sum_g\sum_{j\in J_g} t_jf_i(v_j)\chi_h^{-1}(g_j)=\sum_{j=1}^{N+1}t_j(\mathcal{F}_{i,h}\circ \iota)(v_j^{g_j})$, as desired. On the other hand, $(\mathcal{R}\circ \iota)(v_j^g)=\mathbf{e}_g-\frac{1}{|G|}\mathbf{1}$, where $\mathbf{e}_g$ is the standard basis vector in $\mathbb{R}[G]$ and $\mathbf{1}=\sum_{g\in G}\mathbf{e}_g$. For given $v\in G^{\ast (N+1)}$ and each $g$ such that $J_g\neq\emptyset$, we have $\sum_{j\in J_g} t_j(\mathcal{R}\circ \iota)(v_j^g)=\lambda_g\mathbf{e}_g- \frac{\lambda_g}{|G|}\mathbf{1}$. Thus $\sum_{j=1}^{N+1}t_j(\mathcal{R}\circ \iota)(v_j)=\sum_{g\in G} \sum_{j\in J_g} t_j(\mathcal{R}\circ \iota)(v_j^g)=\sum_{g\in G} (\lambda_g\mathbf{e}_g- \frac{\lambda_g}{|G|}\mathbf{1})=(\mathcal{R}\circ \iota)(v)$.\\

	To prove that $N=2m-m'+r-1$ is tight, suppose that $n<N$ and let $f:\Delta_n\rightarrow \mathbb{C}^d\oplus \mathbb{R}^{d'}$. For $W$ as above, the vanishing of all desired coefficients corresponds to a zero of $A:G^{\ast (n+1)}\rightarrow W$. The argument is then along the lines of [23, Theorem 2.4] and [26, Theorem 1]. Since $A$ is affine, $A(v)=0$ means that $0\in \Conv(A(v_1^{g_1}),\ldots, A(v_{n+1}^{g_{n+1}}))$ for some $(g_1,\ldots, g_{n+1})\in G^{\oplus (n+1)}$, and by assumption each such convex hull has dimension at most $n<N$. On the other hand, requiring that $A(v_1^{g_1}),\ldots, A(v_n^{g_{n+1}})$  ``capture the origin" forces their convex hull to have dimension at least $N$, provided the images $f(v_\ell)$ of the vertices of $\Delta_n$ are generic. This can be seen by tallying the number of independent conditions for $A(v)=0$, $v=\sum_\ell t_\ell v_\ell$. First, one has $r-1$ independent linear conditions on the $t_\ell$ themselves because $\sum_{j\in J_g}t_j=\lambda_g=\frac{1}{|G|}$ for all $g\in G$. The vanishing of each Fourier coefficient yields $2m-m'$ additional linearly independent conditions. Namely, $c_{i,h}=0$ means that $\sum_{g\in G}\sum_{\ell\in J_g}t_\ell f_i(v_\ell)\chi^{-1}_{i,h}(g_\ell)=0$. As the $h$ are distinct for each $1\leq i \leq d+d'$, while for each $i>d$ one has in addition that none of the $c_{i,h}$ are conjugate, it follows that there are $m-m'$ complex linearly independent conditions for almost any choice of $f(v_\ell)$.  For generic $f(v_\ell)$, one likewise has an additional $m'$ real independent conditions when $i>d$ and $|h|=2$. \end{proof} 

\subsection{Fourier Generiticity} 

As stated in Theorem 3.1, the ability to annihilate coefficients depends on a particular coordinate decomposition of $f:\Delta_N\rightarrow \mathbb{R}^{2d+d'}$. For our proofs of Theorems \ref{thm1.2} and \ref{thm1.3}, however, we will need to consider arbitrary coordinate decompositions. To remedy this, we set the following notation for any decomposition of $\mathbb{R}^{2d+d'}=\oplus_{i=1}^d U_i \oplus_{j=1}^{d'}L_j$ by pairwise orthogonal coordinate planes $U_1,\ldots, U_d$ and coordinate lines $L_1,\ldots, L_{d'}$. Let 

\begin{equation}\label{eq:3.9} \sigma: \mathbb{R}^{2d+d'}\rightarrow \mathbb{C}^d\oplus \mathbb{R}^{d'} \end{equation}  

\noindent be the orthogonal transformation of $\mathbb{R}^{2d+d'}$ given by coordinate permutation which sends each $U_i$ to the $i$-th coordinate complex plane of $\mathbb{C}^d$ and each $L_j$ to the $j$-th coordinate line of $\mathbb{R}^{d'}$, and let 
 \begin{equation}\label{eq:3.10} f_\sigma=\sigma\circ f: \Delta_N\rightarrow \mathbb{C}^d\oplus \mathbb{R}^{d'}. \end{equation}
	
	It follows that the vanishing of the Fourier coefficients of the $f_\sigma$ for given $\{x_g\}_{g\in G}$ of $\Delta_N$ with pairwise disjoint support is equivalent to annihilation of those for the compositions of $f$ onto the $U_i$ and $L_j$.\\

Theorem \ref{thm3.1}, shows that for generic affine maps, one cannot annihilate ``too many" Fourier coefficients given the dimension of the simplex. We will now state this as a formal criteria which will be used repeatedly in what follows. For the proof of our van Kampen--Flores type theorems of Section 6, we will also need to make sure that ``extra" Fourier coefficients do not vanish when any $f_\sigma$ is restricted to any face of the simplex. Again, Theorem \ref{thm3.1} shows that such a restriction holds generically in the affine setting.

\begin{definition}\label{def1} [Fourier Generiticity] Let $f:\Delta_N\rightarrow \mathbb{R}^{2d+d'}$, where $d+d'\geq 1$ and $d,d'\geq 0$. Let $G=\mathbb{Z}_{r_1}\oplus \cdots\oplus \mathbb{Z}_{r_k}$ and $r=r_1\cdots r_k$. We say that $f$ is $G$--Fourier generic if  \begin{itemize}
 \item for any coordinate decomposition $\mathbb{R}^{2d+d'}=\oplus_{i=1}^d U_i \oplus_{j=1}^{d'}L_j$ and $\sigma$ as in (\ref{eq:3.9}) and \item for any choice of $S_i$ and $T_i$ as in the statement of Theorem \ref{thm3.1},\end{itemize} 
we have the following condition on any face $\Delta_n$ of $\Delta_N$ with
 $n<2m-m'+r-1$: for the restriction $f_\sigma|\Delta_n:\Delta_n\rightarrow \mathbb{C}^d\oplus \mathbb{R}^{d'}$ of $f_\sigma$ to $\Delta_n$, there does not exist any collection $\{x_g\}_{g\in G}\subset \Delta_n$ of points with pairwise disjoint support for which the Fourier coefficients $c_{i,h}$ \ref{eq:2.3} for $f_\sigma|\Delta_n$ vanish for all $h\in S_i$ and all $1\leq i \leq d+d'$. Finally, we say that $f$ is Fourier generic if it is $G$--Fourier generic for any abelian group $G$. 
\end{definition} 

In particular, note that if $f:\Delta_{N-1}\rightarrow \mathbb{R}^{2d+d'}$ is Fourier generic and $N<T(r,2d+d')$, then $f$ does not admit a Tverberg $r$-partition. 

\begin{remark}\label{rem1} Theorem \ref{thm3.1} shows that almost every affine map is Fourier generic. In the topological setting, the generic designation remains appropriate since again the vanishing of prescribed coefficients is equivalent to a zero of the continuous (and $G$-equivariant) map $A: G^{\ast (n+1) }\rightarrow \mathbb{R}^N$. As $G^{\ast (n+1)}$ is $n$-dimensional, the number of independent conditions again exceeds the degrees of freedom when $n<N$ and so will not vanish for typical $f$. \end{remark}

\section{Proofs in the Affine Setting}  

Given Lemma \ref{lem2.1} and Theorem \ref{thm3.1}, Theorem \ref{thm1.2} follows easily. 

\begin{proof} [Proof of Theorem \ref{thm1.2}.]  For both (a) and (b), suppose that $f: \Delta_{N-1}\rightarrow \mathbb{R}^d$ is Fourier generic and affine, with $d=2k$ or $d=2k+1$ respectively. 

	For (a), let $r=r_1\cdots r_k$, $r_i\geq 3$ and let $G=\oplus_{i=1}^k \mathbb{Z}_{r_i}$.  Given $k$ pairwise orthogonal planes $U_1,\ldots, U_k$, let $\sigma: \oplus_{i=1}^kU_i \rightarrow \mathbb{C}^k$ be the resulting orthogonal transformation (\ref{eq:3.9}) and consider the corresponding  $f_\sigma: \Delta_{N-1}\rightarrow\mathbb{C}^k$ (\ref{eq:3.10}). For $\{x_g\}_{g\in G}$  with pairwise disjoint support, $\{f(x_g)\}_{g\in G}$ will be the vertex set of a $P(r,2k)$ with each $P_{r_i}$ parallel to $U_i$ iff $\{f_\sigma(x_g)\}_{g\in G}$ is the vertex set of a $P(r,2k)$ with each $P_{r_i}$ parallel to $i$-th coordinate plane. Following Lemma \ref{lem2.1}(a), we seek some $\{x_g\}_{g\in G}$ so that the $r-2$ coefficients $c_{i,h}$ given by $h\in G-\{0, \mathbf{e}_i\}$ vanish for each $1\leq i \leq k$ in the Fourier decompositions (\ref{eq:2.2}) of $f_\sigma$. This is guaranteed by Theorem \ref{thm3.1} with $N=2k(r-2)+r=T(r,2k)-2k$, and for this $N$ no additional coefficients can vanish because $f$ is assumed Fourier generic. On the other hand, it is immediate from the optimality of Theorem \ref{thm3.1} and the Fourier characterization of the $P(r,2k)$--partitions from Lemma \ref{lem2.1}(a) that this $N$ is optimal. In particular, $N_{P_r}=T(r,2)-2$ for all $r\geq 3$.\\

	The proof of (b) is nearly identical, except in dealing with the final coordinate of the analogous $f_\sigma: \Delta_{N-1}\rightarrow \mathbb{C}^k \oplus \mathbb{R}$. Considering $G'=G\oplus\mathbb{Z}_2$, we seek $c_{i,h}=0$ for all $h\in G'-\{0, \mathbf{e}_i\}$, $1\leq i \leq k+1$. As $(f_\sigma)_{k+1}$ is real--valued and none of the $h_{k+1,j}$ considered have order 2, we only need to annihilate half of these $2r-2$ Fourier coefficients. Thus all of our desired coefficients vanish provided $N=2(k(2r-2)+\frac{2r-2}{2})+2r=T(2r,2k+1)-(2k+1)$. Generiticity of $f$ again ensures that $c_{i,h_{i,\mathbf{e}_i}}\neq 0$ for all $1\leq i \leq k+1$, 
and the optimality of $N$ is again guaranteed by Theorem \ref{thm3.1} and Lemma \ref{lem2.1}(b).
\end{proof}

As a further application of Theorem \ref{thm3.1}, we give two extensions of  Theorem \ref{thm1.1} to higher dimensions. 

\begin{theorem}\label{thm4.1}$\newline$ 
(a) Let $S$ be a set of $T(r,d)-2$ generic points in $\mathbb{R}^d$. Then for any plane $U$ in $\mathbb{R}^d$, there exists a $P_r$--partition of $S$ with $P_r$ parallel to $U$.\\
(b) Let $d$ be even. Then almost any $T(r,d)-d$ points in $\mathbb{R}^d$ can be $P_r$--partitioned so that $P_r$ lies in a complex 1-flat. In particular, \begin{equation} \label{eq:4.1} N_{(P_r;d)}\leq (r-2)(d+1)+2.\end{equation}
\end{theorem}

\begin{proof}  Again assume that $f:\Delta_{N-1}\rightarrow\mathbb{R}^d$ is affine and Fourier generic. Given a linear 2--flat $U$, let $f_\Phi=\Phi \circ f:\Delta_{N-1}\rightarrow \mathbb{C}\oplus \mathbb{R}^{d-2}$, where $\Phi$ is some orthogonal transformation sending $U$ to the first complex coordinate plane. Considering the resulting Fourier decompositions for $f_\Phi$ with $G=\mathbb{Z}_r$ and letting $N=T(r,d)-2$, Theorem \ref{thm3.1} guarantees some $\{x_g\}_{g\in \mathbb{Z}_r}$ so that $c_{1,g}=0$ for all $g\in \mathbb{Z}_r-\{0,1\}$, and moreover that $c_{i,g}=0$ for each $g\neq0$ for all $i\geq 2$. If $c_{1,1}=0$ also, then $\{x_g\}_{g\in\mathbb{Z}_r}$ would be a Tverberg $r$-partition for $f_\Phi$ and hence for $f$ as well. This is impossible, however, because $N<T(r,d)$ and $f$ is Fourier generic. Thus the $f_\Phi(x_g)$ are the vertices of a regular $r$-gon parallel to the first coordinate plane, and equivalently the $f(x_g)$ are the vertices of a regular $r$--gon parallel to $U$. 
	
	For (b), let $N=T(r,2d)-2d$ and consider $f:\Delta_{N-1}\rightarrow\mathbb{C}^d$. Given $\{x_g\}_{g\in \mathbb{Z}_r}$, annihilating $c_{i,g}$ for all $g\in \mathbb{Z}_r-\{0,1\}$ and all $1\leq i\leq d$  gives $f(x_g)=\mathbf{c}_0+\zeta_r^g\mathbf{c}_1$ for each $g\in \mathbb{Z}_r$, where $\mathbf{c}_j=(c_{1,j},\ldots, c_{d,j})\in \mathbb{C}^d$ for $j=0,1$. \end{proof} 

 That the $N$ of Theorem \ref{thm4.1} (a) (respectively, (b)) is tight again follows from Theorem \ref{thm3.1} and the partition's equivalent Fourier characterization. As with Theorem \ref{thm1.2}, however, the upper bound (\ref{eq:4.1}) is not tight, and indeed subtraction of the remaining $2(d-2)=\dim G_2(\mathbb{R}^d)$ degrees of freedom from the $N$ of part (a) yields the expected value $N_{(P_r;d)}=(r-3)(d+1)+5$ for all $d\geq 2$. This conjectured value is confirmed for all $r=3$ in Section 7 (Proposition \ref{prop7.1}).  
 
  \section{Proofs in the Continuous Setting}

For general abelian groups, care must be taken to ensure vanishing of the desired Fourier coefficients when $f$ is continuous (see, e.g., [25, Theorem 3.2]).  As with the Topological Tverberg theorem, however, standard equivariant cohomological techniques ensure that coefficients can be annihilated as freely as with affine maps if $G=\mathbb{Z}_p^{\oplus k}$ and $p$ is prime.

\begin{proposition}\label{prop5.1}$\newline$
(a) Let $f:\Delta_N\rightarrow \mathbb{C}^d$ and let $G=\mathbb{Z}_p^{\oplus k}$, $p$ an odd prime. For each $1\leq i \leq d$, let $S_i\subseteq G-\{0\}$. Let $m_i=|S_i|$ and let $m=m_1+\cdots +m_d$. If $N=2m+p^k-1$, then there exists some $\{x_g\}_{g\in G}\subset \Delta_N$ from pairwise disjoint faces such that $c_{i,h}=0$ for all  $h\in S_i$ and all $1\leq i \leq d$ in the Fourier expansions (\ref{eq:2.2}).\\
(b)  Let $f:\Delta_N\rightarrow \mathbb{R}^d$ and let $G=\mathbb{Z}_2^{\oplus k}$. For each $1\leq i \leq d$, let $S_i\subseteq  G-\{0\}$. Let $m_i=|S_i|$ and $m=m_1+\cdots +m_d$. If $N=m+2^k-1$, then there exists some $\{x_g\}_{g\in G}\subset \Delta_N$ from pairwise disjoint faces such that $c_{i,h}=0$ for all  $h\in S_i$ and all $1\leq i \leq d$ in the Fourier expansions (\ref{eq:2.2}).
\end{proposition} 

Although Proposition \ref{prop5.1} can be proven as a consequence of a slightly more general version of [25, Theorem 3.2], we shall derive it instead from a lemma of Volovikov [29, Lemma 8]. 

\begin{lemma}\label{lem5.2} Let $G=\mathbb{Z}_p^{\oplus k}$, $p$ prime, and let $X$ and $Y$ be fixed point free $G$--spaces. If $Y$ is a finite--dimensional cohomology $n$--sphere over the field $\mathbb{F}_p$ and $\tilde{H}^i(X; \mathbb{F}_p)=0$ for all $1\leq i\leq n$, then there is no continuous $G$--equivariant map $h:X\rightarrow Y$.
\end{lemma}

Using Proposition \ref{prop5.1}, one can prove all prime power cases of Theorem \ref{thm1.3}(a) using the same argument as for Theorem \ref{thm1.2}.  Namely, for $f: \Delta_N\rightarrow \mathbb{R}^{2d}$ one annihilates all relevant Fourier coefficients of the map $f_\sigma: \Delta_N\rightarrow \mathbb{C}^k$, and as there Fourier generiticity is enough to guarantee that no additional coefficients vanish. The argument for Theorem 1.3(b) is entirely analogous: 

\begin{proof}[Proof of Theorem 1.3(b)] Given $f:\Delta_{N-1}\rightarrow \mathbb{R}^k$ and $N=T(2^k,k)-k$, there exists some $\{x_g\}_{g\in \mathbb{Z}_2^{\oplus k}}$ with pairwise disjoint support for which $c_{i,g}=0$ all $g\in \mathbb{Z}_2^{\oplus k} -\{0, \mathbf{e}_i\}$ for each $1\leq i \leq k$. On the other hand, no other Fourier coefficients vanish by generiticity. Thus each $\{f_i(x_g)\}_{g\in \mathbb{Z}_2^{\oplus k}}$ is the set of endpoints of a segment in $\mathbb{R}$, hence the $f(x_g)$ are the vertex set of a $k$-orthotope with edges parallel to the coordinate axes. \end{proof}

Proposition \ref{prop5.1} also gives immediate topological extensions of Theorem \ref{thm4.1} for odd primes, again with the same proofs as in the affine setting. 

\begin{theorem}\label{thm5.2} Let $f:\Delta_{N-1}\rightarrow \mathbb{R}^d$  be a Fourier generic continuous map and let $r$ be an odd prime.\\
(a) Let $N=T(r,d)-2$. Then for any plane $U$ in $\mathbb{R}^d$, there exists points $x_1,\ldots, x_r\in \Delta_{N-1}$ with pairwise disjoint support such that $f(x_1),\ldots, f(x_r)$ are the vertices of a regular $r$-gon parallel to $U$.\\
(b) If $d$ is even and $N=T(r,d)-d$, then there exist points $x_1,\ldots, x_r\in \Delta_{N-1}$ with pairwise disjoint support such that $f(x_1),\ldots, f(x_r)$ are the vertices of a regular $r$-gon which lie in a complex 1--flat.
\end{theorem} 

\begin{proof}[Proof of Proposition \ref{prop5.1}] 

We shall use the other main configuration space used for Tverberg problems, namely the deleted product 
		 \begin{equation}\label{eq:5.1} (\Delta_N)^{\times G}_\Delta = \{(x_g)_{g\in G}\mid  x_g\in \sigma_g\,\, \forall g\in G,\,\,\, \text{and}\,\,\, \sigma_g\cap \sigma_{g'}=\emptyset\,\, \forall\,\, g\neq g'\}. \end{equation}  
		 
For consistency of notation of notation with [25], we let $G=\mathbb{Z}_p^{\oplus k}$ act freely on $(\Delta_N)^{\times G}_\Delta$ by \textit{left} translations, so that $g'\cdot (x_g)_{g\in G} =(x_{g'+g})_{g\in G}$. It is a crucial fact (first proved in [4]) that $(\Delta_N)^{\times G}_\Delta$ is $(N-|G|+1)$-dimensional and $(N-|G|)$--connected. It therefore has vanishing reduced cohomology $\tilde{H}^i(X;\mathbb{Z}_p)$ for all $1\leq i \leq N-|G|$ by the Hurewicz theorem. For (a), let $f:\Delta_N\rightarrow \mathbb{C}^d$ and $p$ odd. As in the proof of Theorem \ref{thm3.1} (or [25]), let $\mathcal{F}: (\Delta_N)^{\times G} \rightarrow \mathbb{C}^m$ be given by the evaluation of the given Fourier coefficients, i.e., $x=(x_g)_{g\in G} \mapsto \sum_{g\in G} \frac{1}{|G|} f(x_g)\chi_{i,h}^{-1}(g)$ for each $h\in S_i$ and each $1\leq i\leq d$.  As before, this map is $G$-equivariant map with respect to the linear $G$-action on $\mathbb{C}^m$ given by $\oplus_{i,h}\chi_h$. If no zero for this map exists, then $x\mapsto \mathcal{F}(x)/\|\mathcal{F}(x)\|$ would define a $G$-equivariant map from $(\Delta_N)^{\times G}$ to the $(2m-1)$-dimensional unit sphere $Y=S(\mathbb{C}^m)$, violating Volovikov's Lemma for $N=2m+|G|-1$. Thus we have the desired $\{x_g\}_{g\in G}$. The proof of (b) is identical, with $f:\Delta_N\rightarrow \mathbb{R}^d$ and $G=\mathbb{Z}_2^{\oplus k}$.\end{proof} 
		
For the $k=1$ and $r=4$ case of Theorem \ref{thm1.3}, we use [25, Theorem 3.2] with $d=1$. As applied to $P_r$-partitions, this asserts that if the polynomial $q(y)=(r-1)!y^{r-2}$ is non-zero in $\mathbb{Z}[y]/(ry)$, then there exists some $\{x_g\}_{g\in \mathbb{Z}_r}$ from pairwise disjoint faces such that $c_{1,i}=0$ for all $2\leq i<r$ in the Fourier expansion (\ref{eq:2.2}) of the given $f: \Delta_{3r-5}\rightarrow \mathbb{C}$. Clearly, $q(y)\neq 0$ for non--prime $r$ iff $r=4$. 

\section{Two Constrained Versions}

The ``constraint" method of Blagojevi\'c, Frick, and Ziegler [7]  has proven to be a powerful tool for producing a number of interesting variants of the affine and topological Tverberg theorems with surprising ease. In particular, it was crucial in demonstrating counterexamples to the Topological Tverberg Conjecture for non prime powers [6, 12]. We give two example applications of this method to our schema. 

\subsection{van Kampen--Flores type theorems} 

	Given a continuous map $f:\Delta_N\rightarrow \mathbb{R}^d$, one may seek a $r$--Tverberg partition for which each pairwise disjoint face lies in the $k$--skeleton $\Delta_N^{(k)}$ of the simplex. Such dimensionally constrained versions of Tverberg's Theorem were first given in the continuous setting by van Kampen [14] and Flores [11] when $r=2$, extended to all prime powers $r$ and more general $j$--wise intersection types in [22, 30], and subsequently sharpened in [7].  For our purposes, we only consider the following [7, Theorem 6.3]:
	
\begin{theorem}\label{thm6.1} Let $rk\geq d(r-1)$ and $N=(r-1)(d+2)$. Then for any continuous map $f:\Delta_N\rightarrow \mathbb{R}^d$, there exists $r$ disjoint faces $\sigma_1,\ldots, \sigma_r$  with $\dim\sigma_i\leq k$ for all $1\leq i \leq r$ such that $\cap_{i=1}^rf(\sigma_i)\neq \emptyset$. \end{theorem}

Theorem \ref{thm6.1} follows directly from the Topological Tverberg theorem by considering an appropriate ``constraint" function, here the distance map $d: \Delta_N\rightarrow \mathbb{R}$ given by $d(x)=\dist(x,  \Delta_N^{(k)})$. For $f\oplus d: \Delta_{(r-1)(d+2)}\rightarrow \mathbb{R}^d\oplus \mathbb{R}$, there exists  $x_1\in \sigma_1,\ldots, x_r\in \sigma_r$ with the $\sigma_i$ disjoint so that both $f(x_i)$ and $d(x_i)$ are constant. As $r(k+2)>N$, the pigeon hole principle shows that at least one $\sigma_i$ is from $\Delta_{N-1}^{(k)}$, and since $d$ is constantly zero so are all the others. The necessity of $rk\geq d(r-1)$ follows from the optimality of Tverberg's theorem, since one has in particular a Tverberg $r$--partition for $f$ restricted to some $\Delta_{(k+1)r-1}$ face.\\

In even dimensions, the same method produces dimensionally restricted topological $P_r$--partitions when the dimension of the simplex of Theorem \ref{thm6.1} is lowered by $d$: 

\begin{theorem}\label{thm6.2} Let $r$ be an odd prime, let $d$ be even, and let $N=(r-2)(d+2)+2$. Suppose further that $d(r-2)\leq rk <d(r-1)$. If $f: \Delta_N\rightarrow \mathbb{R}^d$ is a  Fourier generic continuous map, then there exists $x_1\in \sigma_1,\ldots, x_r\in \sigma_r$  from pairwise disjoint faces with $\dim \sigma_i\leq k$ for all $1\leq i \leq r$ such that $f(x_1),\ldots, f(x_r)$ are the vertices of a regular $r$-gon.
\end{theorem}

In particular, Theorem \ref{thm6.2} holds for almost every affine map. Letting $d=r-1$ and $k=r-2$, we see that $N=T(r,r-1)-1$ matches that of the Topological Tverberg theorem:

\begin{corollary}\label{cor6.3} Let $r$ be an odd prime. For any Fourier generic continuous map $f:\Delta_{r(r-1)}\rightarrow \mathbb{R}^{r-1}$, there exist $x_1\in \sigma_1,\ldots, x_r\in \sigma_r$  from pairwise disjoint faces with $\dim \sigma_i\leq r-2$ for all $1\leq i \leq r$ such that $f(x_1),\ldots, f(x_r)$ are the vertices of a regular $r$-gon.
\end{corollary}

 In the affine setting, Corollary \ref{cor6.3} states that for almost any set of size $T(r,r-1)$ in $\mathbb{R}^{r-1}$, it is possible to remove a single point so that the remaining $T(r,r-1)-1$ points can  be $P_r$--partitioned by subsets of at most $r-1$ points each.  A picture of this situation when $r=3$ is given below, with the dashed lines indicating a full Tverberg partition of 7 points in the plane: 
 
 \vspace*{-.05in}
 
\begin{center} \includegraphics[scale=.45]{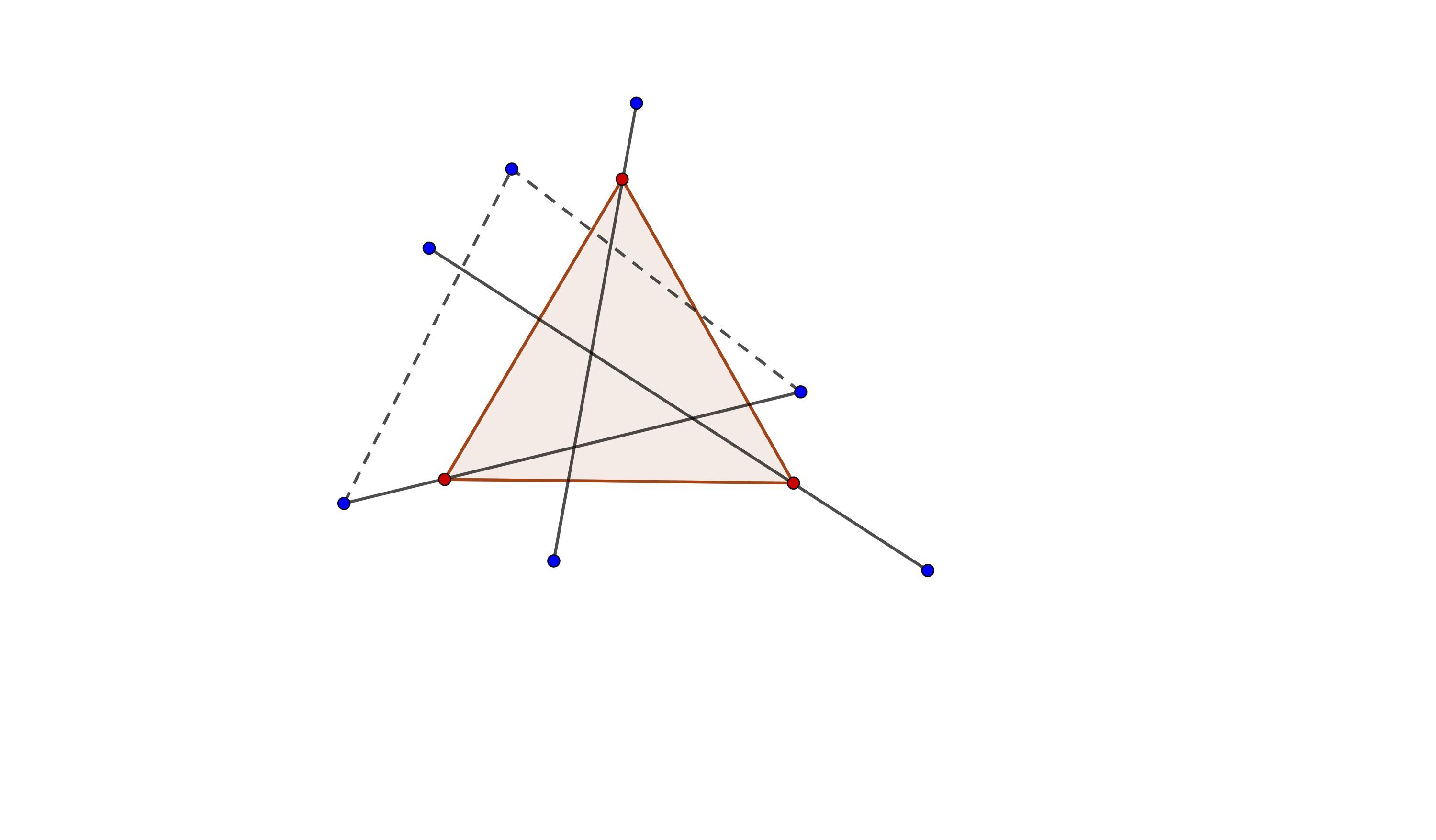} \end{center} \vspace*{-.5in} 

\begin{proof}[Proof of Theorem \ref{thm6.2}] Let $f:\Delta_N\rightarrow \mathbb{C}^d$ be Fourier generic. Letting $f_{d+1}: \Delta_N\rightarrow \mathbb{R}$ be the distance map $f_{d+1}(x)=\dist(x,\Delta_N^{(k)})$ above, again consider $f\oplus f_{d+1}: \Delta_N\rightarrow \mathbb{C}^d\oplus \mathbb{R}$. For $N=2d(r-2)+2\cdot \frac{r-1}{2}+r-1=(r-2)(2d+2)+2$,  Proposition \ref{prop5.1} applied to $G=\mathbb{Z}_r$ guarantees some $\{x_g\}_{g\in \mathbb{Z}_r}$ from pairwise disjoint $\sigma_g$ such that (1) $c_{i,h}=0$ for all $h\in \mathbb{Z}_r-\{0,1\}$ for all $1\leq i \leq d$, as well as (2) that $c_{d+1,h}=0$ for all $h\in \mathbb{Z}_r-\{0\}$. Condition (2) guarantees that $f_{d+1}(x_g)$ is constant. Again one has $N<r(k+2)$ (since $rk\geq 2d(r-2)$), so as before each $\sigma_g$ comes from the $k$--skeleton. As in the proof of Theorem \ref{thm4.1}(a), (1) implies that the $f(x_g)$ 
 are the vertex set of a regular $r$--gon (and in fact one lying in a complex 1--flat), provided at least one of the $c_{i,1}$ is non--zero, $1\leq i \leq d$. On the other hand, the $x_g$ reside in some $\Delta_n$, $n=r(k+1)$. As $rk <2d(r-1)$, $n<T(r,2d)-1$, and so by Fourier generiticity $f$ does not admit a $r$--Tverberg partition. Thus $c_{i,1}\neq 0$ for some $1\leq i \leq d$.\\ \end{proof} 
 
 \subsubsection{Relation to the square peg problem and its variants} As observed by the anonymous referee, the partitions of Corollary \ref{cor6.3} above bear a certain similarity to the famous square peg problem originally posed by Topelitz in 1911 [27] which asks if any simple closed curve in the plane contains the vertices of a square, as well as to variants of that problem which ask for the  vertices of (a similar copy of) some prescribed convex $d$--polytope given any closed $(d-1)$--manifold in $\mathbb{R}^d$ (see, e.g., the review [16]).
   
  In particular, the $r=3$ case of Corollary \ref{cor6.3} can be compared to the fact that any simple closed curve in the plane inscribes any prescribed triangle (up to similarity), and in fact infinitely many such triangles [18]. For the expected value $N=10$, cohomological considerations show that our methods do not allow for a $P_4$--partition of any $f:\Delta_{10}\rightarrow \mathbb{R}^2$ in which each vertex of the square arises from the 1--skeleton. This matches the fact that the square peg problem remains unresolved in the purely continuous setting. As with that problem (see, e.g. [17]), we can show nonetheless that if $f:\Delta_{10} \rightarrow \mathbb{R}^2$ is Fourier generic, then there exist 4 points form $\Delta_{10}^{(1)}$ with pairwise disjoint support whose images are the vertices of a rectangle, and in fact one whose sides can be prescribed parallel to the coordinate axes. To see this, extend $f$ to $f\oplus d: \Delta_{10}\rightarrow \mathbb{R}^2\oplus \mathbb{R}$ as in the proof of Theorem \ref{thm6.2}, where $d$ gives the distance to the 1--skeleton. As there and as in the proof of Theorem \ref{thm1.3}(b), it follows that there exists $\{x_g\}_{g\in \mathbb{Z}_2^{\oplus 2}}\subset \Delta_{10}^{(1)}$ such that each $1\leq i \leq 2$ one has $c_{i,h}=0$ for all $h\neq 0, \mathbf{e}_i$ in the Fourier expansion \eqref{eq:2.2} of $f$. As $f$ is Fourier generic, restricting to the $\Delta_7$ face containing $\{x_g\}_{g\in G}$ shows that $c_{i,\mathbf{e}_i}\neq0$ for each $i=1,2$. Thus the rectangle determined by the $\{f(x_g)\}_{g\in G}$ is non--degenerate.

\subsection{Colored versions with equal barycentric coordinates} 

Another type of variants on Tverberg's theorem considers a coloring of the vertices of the simplex $\Delta_{rn-1}$  by $n$ color classes $C_1,\ldots, C_n$ with $r$ points each. Given a map $f: \Delta_{rn-1}\rightarrow \mathbb{R}^d$, one seeks a Tverberg $r$--partition $\sigma_1,\ldots, \sigma_r$ so that the vertices of each $\sigma_i$ consists of a single point from each color class (such $\sigma_i$ are called ``rainbow" faces). The celebrated B\'ar\'any--Larman conjecture [3], which concerns the classical case $n=d+1$, claims that such ``colorful" Tverberg $r$--partitions exist for all affine maps, and likewise in the continuous setting [31]. This has been verified in the affine cases for all $r$ provided $d=2$ [3], and topologically for all $d$ if $r+1$ is  prime [8].

	In the affine setting [26], Sober\'on proved a variant of this conjecture with the additional condition of equal barycentric coordinates [26, Theorem 1]. This was subsequently recovered in [7], where it was extended to the continuous realm for all prime powers [7, Theorem 8.3 and 8.1]. We give this as the unified statement Theorem \ref{thm6.4} below. By equal barycentric coordinates, one means the following: Let $\Delta_{rn-1}$ be partitioned by $n$ color classes $C_1,\ldots, C_n$ of $r$ points each. Given pairwise disjoint rainbow faces $\sigma_1,\ldots, \sigma_r$, let   $\{v_j^1,\ldots, v_j^n\}$ denote the vertex set of $\sigma_j$, $v_j^i\in C_i$. Each $x_j\in \sigma_j$ is the unique convex sum $x_j=\sum_{i=1}^n t_{i,j} v_j^i$ of one vertex from each color class, and these $t_{i,j}$ are called the barycentric coordinates of $x_j$. One says that $x_1\in \sigma_1,\ldots, x_r\in \sigma_r$ have equal barycentric coordinates provided that $t_{i,j}=t_i$ is independent of $j$ for each $1\leq i \leq n$. 
	 
\begin{theorem}\label{thm6.4} Let $r\geq 2$, $n=(r-1)d+1$, and $N=rn-1$. Suppose that the vertex set of $\Delta_N$ are partitioned by $n$ color classes $C_1,\ldots, C_n$ with $r$ points each. Then for any affine map $f:\Delta_N\rightarrow \mathbb{R}^d$, there exist points $x_1,\ldots, x_r$ with equal barycentric coordinates from pairwise disjoint rainbow faces $\sigma_1,\ldots, \sigma_r$ such that  $f(x_1)=f(x_2)=\cdots = f(x_r)$. This result also holds for continuous maps provided $r+1$ is a prime power. 
\end{theorem} 

Using an argument analogous to that of Sarkaria's Linear Borsuk--Ulam, it was shown in [26] that partitions as in Theorem \ref{thm6.4} do not exist for almost any affine map $f:\Delta_{rn-1} \rightarrow \mathbb{R}^d$ if $n<(r-1)d+1$. As with our proof of Fourier generiticity, one sees by the same observations as in Remark \ref{rem1} that the lack of these partitions for continuous maps in these dimensions is also typical. We shall therefore call such continuous maps \textit{Sober\'on generic}. Removing $dr$ color classes in Sober\'on's result, we have the following $P_r$--variant in even dimensions:

\begin{theorem}\label{thm6.5} Let $r\geq 3$, $n=(r-2)d+1$, and $N=rn-1$, where $d$ is even. Suppose that the vertex set of $\Delta_N$ is partitioned by $n$ color classes $C_1,\ldots, C_n$ with $r$ points each. Then for almost any affine map $f:\Delta_N\rightarrow \mathbb{R}^d$, there exist points $x_1, \ldots, x_r$ with equal barycentric coordinates from pairwise disjoint rainbow faces $\sigma_1,\ldots, \sigma_r$ so that $f(x_1),\ldots, f(x_r)$ are the vertices of a regular $r$-gon. This also holds for Sober\'on generic continuous maps $f$ when $r$ is an odd prime.
\end{theorem}

\begin{proof}[Proof of Theorem \ref{thm6.5}] As with Theorem \ref{thm6.2}, we follow the proof of Theorems 8.1 and 8.3 of [7] using constraints. Let $v_1,\ldots, v_{N+1}$ denote the vertices of $\Delta_N$. For each $1\leq i \leq n$ and $x=\sum_{j=1}^{N+1} t_jv_j$, define $a_i: \Delta_N\rightarrow \mathbb{R}$ by $a_i(x)=\sum_{v_j\in C_i} t_j$. For given $\{x_g\}_{g\in \mathbb{Z}_r}$ from pairwise disjoint $\sigma_g$, the argument there shows that if the $a_i(x_g)$ are constant for each $2\leq i \leq n$, then (1) each $\sigma_g$ is a $n$--rainbow face and (2) the barycentric coordinates of the $x_g$ are all equal (with $t_{i,g}=a_i(x_g)$). Thus one needs to annihilate $r-1$ coefficients for each $a_i$, while for $f:\Delta_N\rightarrow \mathbb{C}^d$, we prescribe $c_{i,g}=0$ for all $g\in \mathbb{Z}_r-\{0,1\}$ and all $1\leq i \leq d$ as before. This can be guaranteed for all $r\geq 3$ in the affine case when $N=2d(r-2)+(n-1)(r-1)+r-1=rn-1$ by Theorem \ref{thm3.1}, as well as for odd primes in the continuous cases by Proposition \ref{prop5.1}. As $n<2d(r-1)+1$, the $f(x_g)$ do not collapse to a single point if they are assumed Sober\'on generic, hence are the vertices of some $P_r$ which as before lies in some complex 1--flat.  \end{proof}

Note that $n=d+1$ in Theorem \ref{thm6.5} when $r=3$. This matches that of the B\'ar\'any--Larman conjecture, which for $r=3$ remains open for all $d>2$ and open in the topological setting if $d=2$. An illustration of our affine planar version is given below: 

\begin{center} \includegraphics[scale=.525]{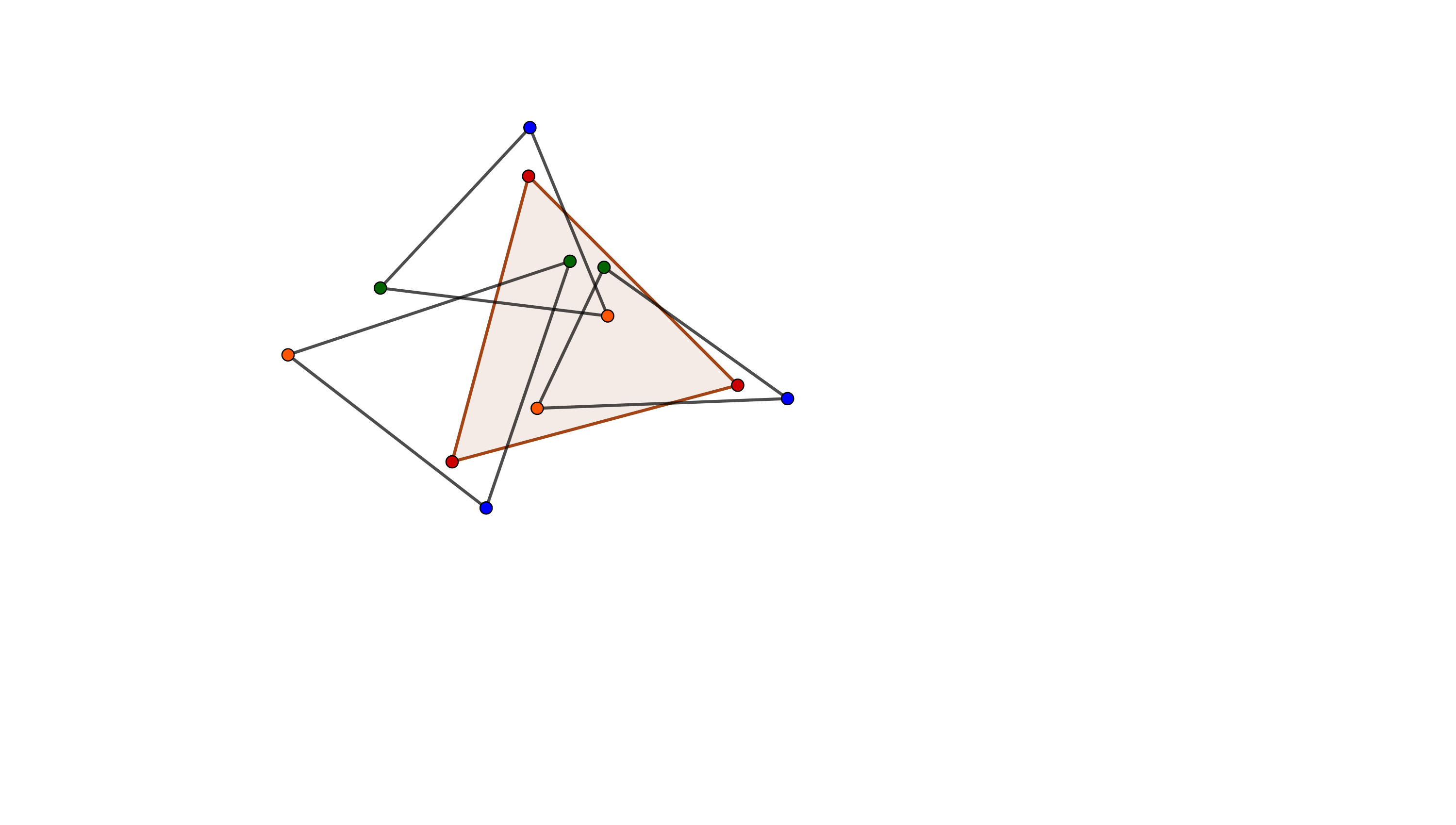}\end{center} 
\vspace*{-.75 in}

\section{Question 1 for Regular Simplicies} 

We conclude with a consideration of Question \ref{quest1} when $P(r,r-1)=\Delta_{r-1}$ is itself a regular $(r-1)$--simplex. As observed by Florian Frick, standard constructions for Tverberg theorems yield an exact value when $r=3$: 

\begin{proposition}\label{prop7.1} $N_{(\Delta_2;d)}= 5$ for all $d\geq 2$. \end{proposition} 

Unlike the Topological Tverberg theorem, however, these constructions fail for all $r>3$, including when $r$ is a prime power. Namely, suppose that $N$ and $d$ are arbitrary and that we seek a $\Delta_{r-1}$--partition for some $f:\Delta_{N-1}\rightarrow \mathbb{R}^d$. The standard approach here would be to define $D: (\Delta_{N-1}^{\times r})_\Delta \rightarrow \mathbb{R}^{\binom{r}{2}}$ by $x=(x_1,\ldots, x_r)\mapsto\|f(x_i)-f(x_j)\|$ for each $1\leq i<j\leq k$. Permutation of the indices produces an action of the symmetric group $\mathfrak{S}_r$ on $(\Delta_{N-1}^{\times r})_\Delta$ which is free, as well as an action on $\mathbb{R}^{\binom{r}{2}}=\{(x_{i,j})_{1\leq i<j\leq k}\mid x_{i,j}\in \mathbb{R}\}$ which is not. As usual, one wants the image of $(\Delta_{N-1})^{\times r})_\Delta$ to intersect the thin diagonal $\delta$ of $\mathbb{R}^{\binom{r}{2}}$ consisting of all $x=(x_{i,j})_{1\leq i<j\leq k}$ with equal $x_{i,j}$. Composing $D$ with the projection of $\mathbb{R}^{\binom{r}{2}}$ onto the orthogonal complement $W:=\delta^\perp$ produces a $\mathfrak{S}_r$-equivariant map $\mathcal{D}: (\Delta_{N-1}^{\times r})_\Delta \rightarrow W$, a zero of which represents some $x_1,\ldots, x_r\in \Delta_N$ with pairwise disjoint support for which the distances $\|f(x_i)-f(x_j)\|$ are equal for all $1\leq i <j\leq k$. Thus one has a $\Delta_{r-1}$--partition provided this distance is non--zero and a Tverberg $r$--partition otherwise. In particular, dimensional considerations yield the following: 

\begin{conjecture} Let $r\geq 3$. Then $N_{(\Delta_{r-1};d)}=\binom{r+1}{2}-1$ for all $d\geq r-1$. \end{conjecture}

	The map $\mathcal{D}$ above must vanish when $r=3$ and $N=5$, since in this case $W$ is the standard representation of $\mathfrak{S}_3$ (see, e.g., [19, Corollary 3.4]). The distances $\|f(x_i)-f(x_j)\|$ are all equal and cannot be zero because $T(3,d)>5$, so the $f(x_i)$ are the vertices of a regular 2--simplex. On the other hand, that $N_{\Delta_2;d}\geq5$ follows by considering 4 coplanar points in $\mathbb{R}^d$ and applying $N_{\Delta_2}=5$. Thus  $N_{(\Delta_2;d)}=5$ as well. 
	
	If $N=\binom{r+1}{2}-1$ and $r\geq 4$, however, $\mathfrak{S}_r$--equivariant maps $(\Delta_{N-1}^{\times r})_\Delta \rightarrow W$ without zeros always exist. This can be seen  directly by considering the above construction as applied to almost any affine map $f:\Delta_{N-1} \rightarrow \mathbb{R}^{r-2}$. As $\binom{r+1}{2}-1<T(r,r-2)$, $f$ has no $r$--Tverberg partition, so the vanishing of the resulting $\mathcal{D}$ would yield a regular $(r-1)$--simplex in $\mathbb{R}^{r-2}$.  		
	
\subsection*{Acknowledgements} \noindent The authors thank the anonymous referee for many useful comments which improved the clarity and exposition of the manuscript, as well as for alerting the authors to connections with the square peg problem. The authors likewise thank Florian Frick and Pablo S\'oberon for helpful discussions and suggestions.

\bibliographystyle{plain}

\end{document}